\documentclass[12pt, a4paper]{article}
\usepackage[english]{babel}

\usepackage{upgreek}
\usepackage{commath}
\usepackage{amsmath,amssymb}
\usepackage{amsthm}
\usepackage{booktabs}
\usepackage{bookmark}
\usepackage{xcolor}
\usepackage{afterpage}
\usepackage[thinlines]{easytable}
\usepackage{array}
\usepackage{mathtools} 
\usepackage{fancyhdr}
\usepackage{graphicx}
\usepackage{enumitem} 
\usepackage{cite}
\usepackage{mathrsfs}
\usepackage[export]{adjustbox}
\usepackage{setspace}
\usepackage{emptypage}
\usepackage[margin=2.8cm]{geometry} 
\graphicspath{ {./images/} }
\usepackage[utf8]{inputenc}
\usepackage{epigraph}
\usepackage[nottoc,numbib]{tocbibind} 
\usepackage{sectsty}
\usepackage{thmtools}
\usepackage{theoremref}

\usepackage[numbers,sort&compress]{natbib}
\bibpunct[, ]{[}{]}{,}{n}{,}{,}

\theoremstyle{plain}
\newtheorem{theorem}{Theorem}
\newtheorem{lemma}{Lemma}
\newtheorem{corollary}{Corollary}
\newtheorem{proposition}{Proposition}

\theoremstyle{definition}

\newtheorem{example}{Example}

\theoremstyle{remark}

\newtheoremstyle{colon}%
{}
{}
{\itshape}
{}
{\bfseries}
{:}
{ }
{}

\theoremstyle{colon}
\newtheorem*{indeed}{Indeed}

\usepackage{setspace}
\providecommand{\keywords}[1]{\textbf{{Keywords: }} #1}

\begin{document}
	\title{Note on the Weak Convergence of Hyperplane $\alpha$-Quantile Functionals and Their Continuity in the Skorokhod J1 Topology \footnote{This paper was originally published in Sparago, P. \textit{"Note on the Weak Convergence of Hyperplane $\alpha$-Quantile Functionals and Their Continuity in the Skorokhod J1 Topology"}, Journal of Theoretical Probability, Vol. 38,  17 (2025) \url{https://doi.org/10.1007/s10959-024-01390-w}. This version contains minor corrections to: (i) the statement of Theorem \ref{weak_convergence_BM_alpha_quantile_time} (nondegeneracy, already assumed throughout in the context of Brownian paths, here added for clarity); (ii) a typo in the statement of Theorem \ref{stopping_quantile_continuity} (in the $\sup,\inf$ arguments); (iii) the statements of Proposition \ref{BM_joint_continuity_occupation_time},  Lemma \ref{BM_as_nonzero_alpha_quantile}, and Proposition  \ref{BM_continuity_random_times} (ruling out the trivial case $\gamma =0$).}}
	
	\author{{Pietro Maria Sparago$^1$}
	}
	
	\date{%
		$^1$London School of Economics\\
		\today
	}
	
	\maketitle
	
	\begin{abstract}
		The $\alpha$-quantile $M_{t,\alpha}$ of a stochastic process has been introduced in \cite{Miura} and important distributional results have been derived in \cite{Akahori}, \cite{Dassios_quantileBM} and \cite{Yor_quantileBM}, with special attention given to the problem of pricing $\alpha$-quantile options. We straightforwardly extend the classical monodimensional setting to $\mathbb{R}^d$ by introducing the hyperplane $\alpha$-quantile and we find an explicit functional continuity set of the $\alpha$-quantile as a functional mapping $\mathbb{R}^d$-valued càdlàg functions to $\mathbb{R}$. This specification allows us to use continuous mapping and assert that if a $\mathbb{R}^d$-valued càdlàg stochastic process $X$ a.s. belongs to such continuity set, then $X^n\Rightarrow X$ (i.e. weakly in the Skorokhod sense) implies $M_{t,\alpha}(X^n)\to^\textnormal{w}M_{t,\alpha}(X)$ (i.e. weakly) in the usual sense. We further the discussion by considering the conditions for convergence of a 'random time' functional of $M_{t,\alpha}$, the first time at which the $\alpha$-quantile has been hit, applied to sequences of càdlàg functions converging in the Skorokhod topology. The Brownian distribution of this functional is studied e.g. in \cite{Chaumont_quantilestopping} and \cite{Dassios_quantileBMstopping}. We finally prove the fact that if the limit process of a sequence of càdlàg stochastic processes is a multidimensional Brownian motion with nontrivial covariance structure, such random time functional applied to the sequence of processes converges - jointly with the $\alpha$-quantile - weakly in the usual sense.
	\end{abstract}

	
\section{Introduction}

The set $\mathcal{D}_{\mathbb{R}^d}[0,\infty)$ is the Skorokhod space of $\mathbb{R}^d$-valued càdlàg functions, which are equivalently denoted with $x=(x_t)_{t \geq 0}=((x_{1,t},...,x_{d,t}))_{t \geq 0}$ unless otherwise specified. This set is equipped with the topology J1 induced by its specification of convergence of sequences (see e.g. (\cite{JacodShiryaev}; VI.1.14. p. 328)); if $d=1$ we write $\mathcal{D}[0,\infty)$ for simplicity. We denote with $\|.\|$ the Euclidean norm in $\mathbb{R}^d,d>1$ and with $|.|$ the absolute value. We denote with $x\cdot y=x'y$ the dot product in $\mathbb{R}^d$. For clarity or ease of notation, the Lebesgue measure on $\mathbb{R}$ will be identified with either $\lambda(dt)$ or $dt$. The Skorokhod space is a Polish space, and there is a metric $\delta$ such that $\delta(x^n,x)\to 0$ if and only if $x^n\to x$ in the Skorokhod topology. \textcolor{black}{
	Following the argument of Prokhorov as in (\cite{JacodShiryaev}; VI.\S1c p. 329), $\delta$ is given as following. First, for each $N \in \mathbb{N}$ define the function $k_N$ by $k_N(t):=1$ if $t\leq N$, $k_N(t):=N+1-t$ if $N<t<N+1$ and $k_N(t):=0$ if $t\geq N+1$. Then, for all $\uplambda:[0,\infty)\to [0,\infty)$ continuous strictly increasing with $\uplambda_0=0$ and $\uplambda_t\uparrow \infty$ as $t\uparrow \infty$ set $|||\uplambda|||:=\sup_{s<t}|\textrm{Log}\frac{\uplambda_t-\uplambda_s}{t-s}|$. Finally, for $x,x'\in \mathcal{D}_{\mathbb{R}^d}[0,\infty)$ set
	$$\begin{aligned}
	\delta_N(x,x')&:=\inf_{\uplambda}(|||\uplambda |||+\|(k_Nx)\circ \uplambda-k_Nx'\|_\infty)\\
	\delta(x,x')&:=\sum_{N\in \mathbb{N}}2^{-N}(1 \wedge \delta_N(x,x'))\\
	\end{aligned}$$
	where $\|.\|_\infty$ is the metric of (\cite{JacodShiryaev}; VI.1.2 p. 325) and $k_Nx$ is the product of $k_N$ and $x\in \mathcal{D}_{\mathbb{R}^d}[0,\infty)$.} We denote the jumps of càdlàg functions with $\Delta x_t:=x_t-x_{t^-}$. The set $C_{\mathbb{R}^d}[0,\infty)\subset \mathcal{D}_{\mathbb{R}^d}[0,\infty)$ is the space of $\mathbb{R}^d$-valued continuous functions. Given an abstract probability space $(\Omega,\mathscr{F},P)$, a càdlàg stochastic process is identified with the measurable function $X:(\Omega,\mathscr{F}) \to (\mathcal{D}_{\mathbb{R}^d}[0,\infty),\mathscr{D}_{\mathbb{R}^d}[0,\infty))$ where $\mathscr{D}_{\mathbb{R}^d}[0,\infty)$ are the Borel sets. If $X\sim \sigma W$ where $W$ is a one-dimensional standard Brownian motion and $\sigma>0$, then we will say that  $X$ is a Brownian motion with volatility parameter $\sigma$. We will also consider sequences of càdlàg stochastic processes $\{X^n=(X_t^n)_{t \geq 0},n \in \mathbb{N}\}$ each defined on its own stochastic basis $(\Omega^n,\mathscr{F}^n,(\mathscr{F}_t^n)_{t \geq 0},P^n)$ which may not be the same as that of $X$. We say that $X^n\Rightarrow X$ if $E^{P^n}[f(X^n)]\to E^{P}[f(X)]$ as $n \to \infty$ for all $f:\mathcal{D}_{\mathbb{R}^d}[0,\infty)\to \mathbb{R}$ bounded and continuous in the Skorokhod topology. Given a time horizon $t>0$ and some $\alpha\in (0,1)$ we define the hyperplane $\alpha$-quantile
\begin{equation}
\label{alpha_quantile}
M_{t,\alpha}(X)=\inf\bigg\{y:\frac{1}{t}\int_0^t\mathbf{1}_{\{z\in \mathbb{R}^d:\gamma \cdot z \leq y\}}(X_s)ds\geq \alpha\bigg\}
\end{equation}
for some fixed $\gamma \in \mathbb{R}^d$ and a càdlàg stochastic process $X$. We denote with $\tau_{M_{t,\alpha}}(X)$ the random time at which $X$ has hit $M_{t,\alpha}(X)$. \textcolor{black}{
	We shall give the precise definitions of these functionals respectively in Lemma \ref{alpha_quantile_properties} and Theorem \ref{stopping_quantile_continuity}}. Our approach is to consider the $\alpha$-quantile and its random time as functionals on $\mathcal{D}_{\mathbb{R}^d}[0,\infty)$, and we study their continuity properties in the Skorokhod topology. Quoting (\cite{JacodShiryaev}, VI p. 337):\\

\textit{''It is important to decide whether a function defined on} $\mathcal{D}_{\mathbb{R}^d}[0,\infty)$ {\textit{is continuous for the Skorokhod topology. As a matter of fact, not very many functions are so, hence the question becomes: at which points of}} $\mathcal{D}_{\mathbb{R}^d}[0,\infty)$ {\textit{is a given function continuous?''}}\\

To address this matter, we shall first characterise the (joint) continuity set of $x\mapsto (M_{T,\alpha}(x),\tau_{M_{T,\alpha}}(x))$ in the Skorokhod topology. Our main result is then the following; its corollary is an immediate application to obtain a means convergence result for $\alpha$-quantile option payoffs.

\begin{theorem} 
	\label{weak_convergence_BM_alpha_quantile_time}
	Let $\alpha \in (0,1),T>0,\gamma \in \mathbb{R}^d$. Let $X=\Sigma W$ where $W$ is a standard $\mathbb{R}^d$-Brownian motion and $\Sigma$ is such that $\Sigma \Sigma'$ is a (nondegenerate) variance-covariance matrix. Then, $X$ belongs with probability one to the continuity set of $\mathcal{D}_{\mathbb{R}^d}[0,\infty)\ni x\mapsto (M_{T,\alpha}(x),\tau_{M_{T,\alpha}}(x))$ in the Skorokhod topology. Therefore: suppose $X^n\Rightarrow X$, then $(M_{T,\alpha}(X^n),\tau_{M_{T,\alpha}}(X^n))\to^{\textnormal{w}} (M_{T,\alpha}(X),\tau_{M_{T,\alpha}}(X))$ in the usual sense.
\end{theorem}
\begin{proof} Of course $X$ has a.s. continuous paths starting at zero with probability one. With probability one the conditions (1),(2),(3) of Corollary \ref{stopping_quantile_continuity_C} are satisfied by Lemma \ref{BM_as_nonzero_alpha_quantile}, Proposition \ref{BM_joint_continuity_occupation_time} and Proposition \ref{BM_continuity_random_times}. So we conclude by continuous mapping.
\end{proof}
\textcolor{black}{
	In the mathematical finance literature, functional central limit theorems with weak limit $X=\Sigma W$ emerge when considering the scaling limits of sequences of processes modelling asset prices. An important example is the scaling limit (of differences) of point processes with mutually exciting jump intensities of \cite{Bacry}. Another example is the small-time functional central limit theorem of \cite{Gerholdetal}. Theorem \ref{weak_convergence_BM_alpha_quantile_time} establishes continuous mapping in the case of the hyperplane $\alpha$-quantile jointly with its random time functional when applied to these sequences of processes. At the same time, the corollary of Theorem \ref{weak_convergence_BM_alpha_quantile_time} which we now are going to present also establishes the convergence of option prices when the underlying quantities are the $\alpha$-quantile and its random time functional, given that some technical conditions on $(X^n)_{n \in \mathbb{N}}$ are met. Given a payoff function $g:\mathbb{R}^2\to \mathbb{R}$ applied to $(M_{T,\alpha}(Y),\tau_{M_{T,\alpha}}(Y))$ of a càdlàg stochastic process $Y$, the standard pricing operation would consist in - setting the interest rate to zero for simplicity - computing the expectation $E^{P}[g(M_{T,\alpha}(Y),\tau_{M_{T,\alpha}}(Y))]$ under the appropriate probability measure. This pricing problem has been studied in particular in \cite{Dassios_quantileBMstopping}.}

\begin{corollary}
	\label{linear_growth_functional_convergence_alpha_quantile_time}	
	Let $X^n\Rightarrow X$ and $X^n$ are all martingales in their respective filtrations. Suppose that $f:\mathbb{R}\to \mathbb{R}$ satisfies the growth condition $|f(z)|^p\leq D(1+|z|^p)$ for some $p>1,D>0$. \textcolor{black}{
		Let $v>0$.} Define $g(m,u)=f(m)\mathbf{1}_{(v,\infty)}(u)$ and its continuity set $C(g)$. If 
	\begin{enumerate}
		\item $\sup_{n \in \mathbb{N}}E^{P^n}[|(\gamma \cdot X^n)_T|^p]<\infty$;
		\item $P((M_{T,\alpha}(X),\tau_{M_{T,\alpha}}(X))\in C(g))=1$;
	\end{enumerate}
	Then $E^{P^n}[f(M_{T,\alpha}(X^n))\mathbf{1}_{(v,\infty)}(\tau_{M_{T,\alpha}}(X^n))]\to E^P[f(M_{T,\alpha}(X))\mathbf{1}_{(v,\infty)}(\tau_{M_{T,\alpha}}(X))]$ as $n \to \infty$.
\end{corollary}
\begin{proof} The claim follows - straightforwardly with a standard uniform integrability argument with the de la Vallée Poussin criterion - from Theorem \ref{weak_convergence_BM_alpha_quantile_time}, Lemma \ref{alpha_quantile_bounds}, Doob's $L^p$-inequality of càdlàg submartingales (see e.g. (\cite{KaratzasShreve}; 3.8.(iv) p. 13)), Skorokhod's representation theorem (see e.g. (\cite{Bogachev}; Vol. II, 8.5.4. p. 201)) and continuous mapping. \textcolor{black}{
		Indeed: since $\sup_{s\leq t}|\gamma \cdot x_s|\geq \sup_{s\leq t}(-\gamma \cdot x_s)=-\inf_{s\leq t}(\gamma \cdot x_s)$ and $\sup_{s\leq t}(\gamma \cdot x_s)\leq \sup_{s\leq t} |\gamma \cdot x_s|$, by following Lemma \ref{alpha_quantile_bounds} we have that $|M_{t,\alpha}(x)|^p\leq \sup_{s\leq t}|\gamma \cdot x_s|^p$ for all $x \in \mathcal{D}_{\mathbb{R}^d}[0,\infty)$. We have that $(\gamma \cdot X^n)_{n \in \mathbb{N}}$ are all martingales in the respective filtrations of $(X^n)_{n \in \mathbb{N}}$. Using the growth condition on $f$, then Lemma \ref{alpha_quantile_bounds} and Doob's $L^p$-inequality we obtain
		$$E^{P^n}[|f(M_{T,\alpha}(X^n))|^p\mathbf{1}_{(v,\infty)}(\tau_{M_{T,\alpha}}(X^n))]\leq D\bigg(1+\bigg(\frac{p}{p-1}\bigg)^pE^{P^n}[|(\gamma \cdot X^n)_T|^p]\bigg)$$
		so that $\sup_nE^{P^n}[|f(M_{T,\alpha}(X^n))|^p\mathbf{1}_{(v,\infty)}(\tau_{M_{T,\alpha}}(X^n))]<\infty$ by our first assumption. Since $(M_{T,\alpha}(X^n),\tau_{M_{T,\alpha}}(X^n))\to^{\textnormal{w}} (M_{T,\alpha}(X),\tau_{M_{T,\alpha}}(X))$, Skorokhod's representation allows us to pass to some $(\xi^{1,n},\xi^{2,n})\sim (M_{T,\alpha}(X^n),\tau_{M_{T,\alpha}}(X^n))$ and $(\xi^{1},\xi^{2})\sim (M_{T,\alpha}(X),\tau_{M_{T,\alpha}}(X))$ on the probability space $([0,1],\mathscr{B}[0,1],\lambda)$ such that $(\xi^{1,n},\xi^{2,n})\to^{\lambda\textnormal{-a.e.}}(\xi^{1},\xi^{2})$. Also,
		$$\sup_{n \in \mathbb{N}}E^{P^n}[|f(M_{T,\alpha}(X^n))|^p\mathbf{1}_{(v,\infty)}(\tau_{M_{T,\alpha}}(X^n))]=\sup_{n \in \mathbb{N}}\int_{[0,1]}|f(\xi^{1,n})|^p\mathbf{1}_{(v,\infty)}(\xi^{2,n})d\lambda$$
		and $1=P((M_{T,\alpha}(X),\tau_{M_{T,\alpha}}(X))\in C(g))=\lambda((\xi^{1},\xi^{2})\in C(g))$ by our second assumption, so that by a.e. continuous mapping, uniform integrability and Vitali's convergence theorem we obtain the claim.
	}
\end{proof}

\section{Hyperplane $\alpha$-quantiles and their random time functionals}

In this section we present our results on the continuity sets of the hyperplane $\alpha$-quantile and its random time functional defined as functions on $\mathcal{D}_{\mathbb{R}^d}[0,\infty)$. We shall begin by recalling some elementary results about generalized inverse functions, which we will make use of often.

\begin{proposition}{(see e.g. \cite{McNeiletal}; A.3. p. 641-642)}
	\label{McNeiletal}	
	If $x=(x_t)_{t\in \mathbb{R}}$ is $\mathbb{R}$-valued nondecreasing, not necessarily càdlàg, then the generalized inverse given by $x^{-1}(y):=\inf\{t:x_t\geq y\}$ is such that:
	\begin{enumerate}
		\item $x^{-1}$ is nondecreasing, left-continuous;
		\item $x$ is continuous if and only if $x^{-1}$ is strictly increasing;
		\item $x$ is strictly increasing if and only if $x^{-1}$ is continuous.
	\end{enumerate}
	If $x$ is also right-continuous it follows that $x_t \geq y\iff x^{-1}(y)\leq t$ (equivalently $x_t < y\iff x^{-1}(y)> t$). In fact, right-continuity is needed only for the implication $(\Leftarrow)$.
\end{proposition}

We now introduce the hyperplane $\alpha$-quantile functional on the Skorokhod space, which generalizes its special case found in the literature. We thus reformulate the quantity presented in Equation \ref{alpha_quantile} as a deterministic function acting on the space of $\mathbb{R}^d$-valued càdlàg functions. To identify a continuity set, we employ a proof technique idea found in (\cite{JacodShiryaev}; 2.10. p. 340, 2.11. p. 341) for the continuity of first exit times; this argument is based on dense sets in $\mathbb{R}$. The gist of this technique will be present in later arguments as well.

\begin{lemma} 
	\label{alpha_quantile_properties}
	Let $t>0,\alpha \in (0,1),\gamma \in \mathbb{R}^d$. Define the hyperplane $\alpha$-quantile functional $\mathcal{D}_{\mathbb{R}^d}[0,\infty) \ni x\mapsto M_{t,\alpha}(x)$ by
	\begin{equation}
	M_{t,\alpha}(x)=\inf\bigg\{y:\frac{1}{t}\int_0^t\mathbf{1}_{\{z\in \mathbb{R}^d:\gamma \cdot z \leq y\}}(x_s)ds\geq \alpha\bigg\}
	\end{equation}
	Then:
	\begin{enumerate}
		\item $M_{t,\alpha}$ is a $\mathscr{D}_{\mathbb{R}^d}[0,\infty)/\mathscr{B}(\mathbb{R})$-measurable function;
		\item For $t>0,x \in \mathcal{D}_{\mathbb{R}^d}[0,\infty)$ fixed, $\alpha \mapsto M_{t,\alpha}(x)$ is nondecreasing and left-continuous;
		\item Let $M_{t,\alpha^+}(x):=\lim_{\beta \downarrow \alpha}M_{t,\beta}(x)$. The set $U(x,t)=\{\alpha\in (0,1):M_{t,\alpha}(x)<M_{t,\alpha^+}(x)\}$ is at most countable.
	\end{enumerate}
\end{lemma}

\begin{proof} We first note that $\mathbf{1}_{\{z\in \mathbb{R}^d:\gamma \cdot z \leq y\}}(x_s)=\mathbf{1}_{(-\infty,y]}(\gamma \cdot x_s)$. (1) Let $t>0,x \in \mathcal{D}_{\mathbb{R}^d}[0,\infty)$. Then $y \mapsto \frac{1}{t}\int_0^t\mathbf{1}_{(-\infty,y]}(\gamma \cdot x_s)ds$ is nondecreasing, taking values in $[0,1]$. Also since $\mathbf{1}_{(-\infty,y]}(\gamma \cdot x_s)=\mathbf{1}_{[\gamma \cdot x_s,\infty)}(y)$ it is also right-continuous in $y$, and therefore, $\alpha \mapsto M_{t,\alpha}(x)$ is its generalized inverse function. So we have $M_{t,\alpha}(x)>y\iff \frac{1}{t}\int_0^t\mathbf{1}_{(-\infty,y]}(\gamma \cdot x_s)ds<\alpha$, therefore $M_{t,\alpha}(x)$ is a measurable function - indeed, the time integral functional is a $\mathscr{D}_{\mathbb{R}^d}[0,\infty)/\mathscr{B}(\mathbb{R})$-measurable function. (2) It follows that $\alpha \mapsto M_{t,\alpha}(x)$ is nondecreasing and left-continuous (see Proposition \ref{McNeiletal}). (iii) Since $\alpha \mapsto M_{t,\alpha}(x)$ is nondecreasing, it has at most countably many discontinuities and the claim follows.
\end{proof}

\begin{lemma} 
	\label{continuity_set_continuous_paths}	
	\textcolor{black}{
	}
	$x \in C_{\mathbb{R}^d}[0,\infty)$ implies $\{\alpha\in (0,1):M_{t,\alpha}(x)<M_{t,\alpha^+}(x)\}=\emptyset$, i.e. $\alpha\mapsto M_{t,\alpha}(x)$ for $\alpha \in (0,1)$ is continuous.
\end{lemma}
\begin{proof} Consider $y \mapsto \frac{1}{t}\int_0^t\mathbf{1}_{(-\infty,y]}(\gamma \cdot x_s)ds$. Since $s\mapsto \gamma \cdot x_s$ is continuous in $[0,t]$, it attains both its maximum $\overline{x}:=\sup_{s\leq t}(\gamma \cdot x_s)$ and its minimum $\underline{x}:=\inf_{s\leq t}(\gamma \cdot x_s)$. Without loss of generality, let $\underline{x}<\overline{x}$ (the case $\underline{x}=\overline{x}$ is immediate). We note: $\int_0^t\mathbf{1}_{(-\infty,y]}(\gamma \cdot x_s)ds=\lambda(\{s\in[0,t]:\gamma \cdot x_s\leq y\})$. Let $\varepsilon >0$ and $y \in [\underline{x},\overline{x})$. Indicate with $(\gamma \cdot x)^{-1}$ the inverse image of $s \mapsto \gamma \cdot x_s$. We get $\lambda(\{s\in[0,t]:y<\gamma \cdot x_s\leq y+\varepsilon\})\geq \lambda((\gamma \cdot x)^{-1}((y,y+\varepsilon))\cap[0,t])$. But $(\gamma \cdot x)^{-1}((y,y+\varepsilon))\cap [0,t]$ is non-empty and contains a non-empty open set since $\gamma \cdot x$ is continuous. So $\lambda(\{s\in[0,t]:y<\gamma \cdot x_s\leq y+\varepsilon\})>0$, and we conclude that $y \mapsto \frac{1}{t}\int_0^t\mathbf{1}_{(-\infty,y]}(\gamma \cdot x_s)ds$ is strictly increasing over $[\underline{x},\overline{x})$, which implies (see Proposition \ref{McNeiletal}) our conclusion.
\end{proof}

\begin{proposition}
	\label{alpha_quantile_continuity}
	Let $t>0,\gamma \in \mathbb{R}^d$. The hyperplane $\alpha$-quantile function $x\mapsto M_{t,\alpha}(x)$ is continuous in the Skorokhod topology at all $x$ such that $\alpha \notin U(x,t)=\{\alpha \in (0,1):M_{t,\alpha}(x)<M_{t,\alpha^+}(x)\}$. 
\end{proposition}
\begin{proof} Suppose $x^n\to x$ in the Skorokhod topology and $\alpha \notin U(x,t)$. Then $\gamma \cdot x^n\to \gamma \cdot x$ in the Skorokhod topology. Let $y \in \{y:\lambda(\{t:\gamma \cdot x_t=y\})=0\}$; then $\frac{1}{t}\int_0^t\mathbf{1}_{(-\infty,y]}(\gamma \cdot x_s^n)ds\to \frac{1}{t}\int_0^t\mathbf{1}_{(-\infty,y]}(\gamma \cdot x_s)ds$ in the usual sense. If $M_{t,\alpha}(x)>y$ then $\frac{1}{t}\int_0^t\mathbf{1}_{(-\infty,y]}(\gamma \cdot x_s)ds<\alpha$ and $\frac{1}{t}\int_0^t\mathbf{1}_{(-\infty,y]}(\gamma \cdot x_s^n)ds<\alpha$ for $n$ large enough, therefore $M_{t,\alpha}(x^n)>y$ as well. The set $\{y:\lambda(\{t:\gamma \cdot x_t=y\})=0\}$ is dense in $\mathbb{R}$. We then conclude that $\liminf_{n \to \infty}M_{t,\alpha}(x^n)\geq M_{t,\alpha}(x)$. Since $\alpha \notin U(x,t)$, we get that $y>M_{t,\alpha}(x)$ implies $y>M_{t,\alpha^+}(x)$. Then, for some $\beta>\alpha$ we have $y\geq M_{t,\beta}(x)$, which implies $\frac{1}{t}\int_0^t\mathbf{1}_{(-\infty,y]}(\gamma \cdot x_s)ds\geq \beta>\alpha$. If $y \in \{y:\lambda(\{t:\gamma \cdot x_t=y\})=0\}$, then $\frac{1}{t}\int_0^t\mathbf{1}_{(-\infty,y]}(\gamma \cdot x_s^n)ds>\alpha$ for $n$ large enough and we conclude $M_{t,\alpha}(x^n)\leq y$. We then have $\limsup_{n \to \infty}M_{t,\alpha}(x^n)\leq M_{t,\alpha}(x)$ and we conclude.
\end{proof}

\begin{theorem}
	\label{weak_convergence_alpha_quantile}	
	Let $\alpha \in (0,1),t>0,\gamma \in \mathbb{R}^d$. Let $X^n\Rightarrow X$. Suppose that for all $\varepsilon >0$ we have, alternatively:
	\begin{enumerate}
		\item $P(M_{t,\beta}(X)>M_{t,\alpha}(X)+\varepsilon)\to 0$ as $\beta \downarrow \alpha$ or,
		\item $P(X \in C_{\mathbb{R}^d}[0,\infty))=1$;
	\end{enumerate}
	Then $M_{t,\alpha}(X^n)\to^{\textnormal{w}}M_{t,\alpha}(X)$ in the usual sense.
\end{theorem}
\begin{proof} We have shown in Proposition \ref{alpha_quantile_continuity} that $x\mapsto M_{t,\alpha}(x)$ is continuous at all $x$ such that $\alpha \notin U(x,t)$, i.e. $x \notin \{x:M_{t,\alpha}(x)<M_{t,\alpha^+}(x)\}$. Recall also that $\beta \mapsto M_{t,\beta}(x)$ is nondecreasing left-continuous by Lemma \ref{alpha_quantile_properties}. If (2) holds, the claim follows by Lemma \ref{continuity_set_continuous_paths}. If (1) holds, without loss of generality let $\beta_k\downarrow  \alpha$ for a sequence $(\beta_k)_{k \in \mathbb{N}}\subseteq (0,1)$. Then for $\varepsilon >0$ we have
	$$
	\{\omega:M_{t,\beta_k}(X(\omega))>M_{t,\alpha}(X(\omega))+\varepsilon\}\supseteq\{\omega:M_{t,\alpha^+}(X(\omega))>M_{t,\alpha}(X(\omega))+\varepsilon\},\,\forall k \in \mathbb{N}
	$$
	which implies $P(M_{t,\alpha^+}(X)>M_{t,\alpha}(X)+\varepsilon)=0,\forall \varepsilon >0$ and so $P(M_{t,\alpha^+}(X)>M_{t,\alpha}(X))=0$. But then, $X$ a.s. belongs to the continuity set we have identified in Proposition \ref{alpha_quantile_continuity}, and by continuous mapping the claim follows.
\end{proof}

\begin{lemma}
	\label{alpha_quantile_bounds}	
	Let $\alpha \in (0,1),t>0,\gamma \in \mathbb{R}^d$ and $x \in \mathcal{D}_{\mathbb{R}^d}[0,\infty)$. Then $\inf_{s\leq t}(\gamma \cdot x_s)\leq M_{t,\alpha}(x)\leq \sup_{s\leq t}(\gamma \cdot x_s)$.
\end{lemma}
\begin{proof} Suppose $M_{t,\alpha}(x)<\inf_{s\leq t}(\gamma \cdot x_s)$. We have $\frac{1}{t}\int_0^t\mathbf{1}_{(-\infty,M_{t,\alpha}(x)]}(\gamma \cdot x_s)ds\geq \alpha>0$ (see Proposition \ref{McNeiletal}). But at the same time $\mathbf{1}_{(-\infty,M_{t,\alpha}(x)]}(\gamma \cdot x_s)=0,\forall s\leq t$ because $M_{t,\alpha}(x)<\inf_{s\leq t}(\gamma \cdot x_s)$, which is a contradiction. Now suppose $M_{t,\alpha}(x)>\sup_{s\leq t}(\gamma \cdot x_s)$. Then, this implies $\frac{1}{t}\int_0^t\mathbf{1}_{(-\infty,\sup_{s\leq t}(\gamma \cdot x_s)]}(\gamma \cdot x_s)ds<\alpha<1$ (see again Proposition \ref{McNeiletal}), but $\mathbf{1}_{(-\infty,\sup_{s\leq t}(\gamma \cdot x_s)]}(\gamma \cdot x_s)=1,\forall s\leq t$ which yields a contradiction. We conclude.
\end{proof}

We now discuss the following random time functional of the hyperplane $\alpha$-quantile - the first time at which this quantity has been hit - and we identify in detail some conditions at which this functional is continuous in the Skorokhod topology.

\begin{theorem}
	\label{stopping_quantile_continuity}
	Let $T>0,\alpha \in (0,1),\gamma \in \mathbb{R}^d$. Let $x^n\to x$ in the Skorokhod topology and $x$ such that $ \alpha \notin U(x,T)$. Define:
	\begin{equation}
	x\mapsto \tau_{M_{T,\alpha}}(x):=\begin{cases}
	\inf\bigg\{t>0:\sup_{s\leq t}(\gamma \cdot x)_s\geq  M_{T,\alpha}(x)\bigg\}&\textrm{ if }\,M_{T,\alpha}(x)>\gamma \cdot x_0\\
	0&\textrm{ if }\,M_{T,\alpha}(x)= \gamma \cdot x_0\\
	\inf\bigg\{t>0:\inf_{s\leq t}(\gamma \cdot x)_s\leq  M_{T,\alpha}(x)\bigg\}&\textrm{ if }\,M_{T,\alpha}(x)< \gamma \cdot x_0\\
	\end{cases}
	\end{equation}
	Suppose that $M_{T,\alpha}(x)\neq \gamma \cdot x_0$. Then:
	\begin{enumerate}
		\item $\liminf_{n}\tau_{M_{T,\alpha}}(x^n)\geq \tau_{M_{T,\alpha}}(x)$;
		\item If, equivalently, $y\mapsto \frac{1}{T}\int_0^T\mathbf{1}_{\{z:\gamma \cdot z\leq y\}}(x_s)ds$ is continuous or $\beta \mapsto M_{T,\beta}(x)$ is strictly increasing and additionally $\alpha\in  \{\alpha:\lim_{\beta \to \alpha}\tau_{M_{T,\beta}}(x)=\tau_{M_{T,\alpha}}(x)\}$, then we have $\tau_{M_{T,\alpha}}(x^n)\to \tau_{M_{T,\alpha}}(x)$.
	\end{enumerate} 
	In fact, if the conditions of (2) are satisfied, then we also have $$(M_{T,\alpha}(x^n),\tau_{M_{T,\alpha}}(x^n))\to (M_{T,\alpha}(x),\tau_{M_{T,\alpha}}(x))$$
	in the usual sense.
\end{theorem}

\begin{proof} First note that, since $x^n\to x$ in the Skorokhod topology, then $x^n_0\to x_0$ and so $\gamma \cdot x^n_0\to \gamma\cdot x_0$. We exlcude the case $M_{T,\alpha}(x)=\gamma \cdot x_0$ by assumption. For the cases $M_{T,\alpha}(x)>\gamma \cdot x_0$ and $M_{T,\alpha}(x)<\gamma \cdot x_0$, we prove the first two claims in order:
	
	\begin{enumerate}
		
		\item From the fact that $x \notin \{x:M_{T,\alpha}(x)<M_{T,\alpha^+}(x)\}$ it follows $M_{T,\alpha}(x^n)\to M_{T,\alpha}(x)$ by Proposition \ref{alpha_quantile_continuity}. Suppose $M_{T,\alpha}(x)<\gamma \cdot x_0$. For all $n$ sufficiently large $M_{T,\alpha}(x^n)<\gamma \cdot x^n_0$. Let $t \notin \{t:|\Delta (\gamma \cdot x)_t|>0\}$; then $\inf_{s\leq t}(\gamma \cdot x^n)_s	\to \inf_{s\leq t}(\gamma \cdot x)_s$. If $t<\tau_{M_{T,\alpha}}(x)$, then $\inf_{s\leq t}(\gamma \cdot x)_s>M_{T,\alpha}(x)$ and for $n$ sufficiently large we have $\inf_{s\leq t}(\gamma\cdot x^n)_s>M_{T,\alpha}(x^n)$ so $t<\tau_{M_{T,\alpha}}(x^n)$ as well for all $n$ large enough. Since $\{t:|\Delta (\gamma \cdot x)_t|>0\}$ is at most countable due to $\gamma \cdot x$ being càdlàg, its complement is dense in $\mathbb{R}$, and we conclude. An analogous argument holds for $M_{T,\alpha}(x)>\gamma \cdot x_0$. 
		
		\item The equivalence of the first conditions can be seen from Proposition \ref{McNeiletal}. Consider $x$ such that $\alpha \in \{\alpha:\lim_{\beta \to \alpha}\tau_{M_{T,\beta}}(x)=\tau_{M_{T,\alpha}}(x)\}$, in addition to $\alpha \notin U(x,T)$. Suppose again $M_{T,\alpha}(x)<\gamma \cdot x_0$. Choose a sequence $\beta_k\uparrow \alpha$. If $t >\tau_{M_{T,\alpha}}(x)$ then $t >\lim_{k}\tau_{M_{T,\beta_k}}(x)$ so $t \geq \tau_{M_{T,\beta_k}}(x)$ for some $\beta_k <\alpha$. Let $t \notin \{t:|\Delta (\gamma \cdot x)_t|>0\}$, so that $\inf_{s\leq t}(\gamma \cdot x^n)_s\to \inf_{s\leq t}(\gamma \cdot x)_s$. We have $\inf_{s\leq t}(\gamma \cdot x)_s\leq M_{T,\beta_k}(x)<M_{T,\alpha}(x)$ (due to strict monotonicity of $\beta \mapsto M_{T,\beta}(x)$) and $\inf_{s\leq t}(\gamma \cdot x^n)_s<M_{T,\alpha}(x^n)$ for all $n$ large enough because $M_{T,\alpha}(x^n)\to M_{T,\alpha}(x)$ since $\alpha \notin U(x,T)$. Therefore, $t \geq \tau_{M_{T,\alpha}}(x^n)$ for all $n$ large enough. Since $\{t:|\Delta (\gamma \cdot x)_t|>0\}$ is at most countable, we conclude that $\limsup_n\tau_{M_{T,\alpha}}(x^n)\leq \tau_{M_{T,\alpha}}(x)$. Now suppose $M_{T,\alpha}(x)>\gamma \cdot x_0$. Let again $t \notin \{t:|\Delta (\gamma \cdot x)_t|>0\}$ so that $\sup_{s\leq t}(\gamma \cdot x^n)_s\to \sup_{s\leq t}(\gamma \cdot x)_s$. Choose a sequence $\beta_k\downarrow \alpha$. If $t >\tau_{M_{T,\alpha}}(x)$ then $t >\lim_{k}\tau_{M_{T,\beta_k}}(x)$ so $t \geq \tau_{M_{T,\beta_k}}(x)$ for some $\beta_k >\alpha$. We have $\sup_{s\leq t}(\gamma \cdot x)_s\geq M_{T,\beta_k}(x)> M_{T,\alpha}(x)$ and $\sup_{s\leq t}(\gamma \cdot x^n)_s>M_{T,\alpha}(x^n)$ for all $n$ large enough. So $t\geq \tau_{M_{T,\alpha}}(x^n)$ for all $n$ large enough. Since $\{t:|\Delta (\gamma \cdot x)_t|>0\}$ is at most countable, we conclude that $\limsup_n\tau_{M_{T,\alpha}}(x^n)\leq \tau_{M_{T,\alpha}}(x)$. We then conclude that $\tau_{M_{T,\alpha}}(x^n)\to \tau_{M_{T,\alpha}}(x)$.
	\end{enumerate}
	The last statement immediately follows from the previous result and the fact that $|M_{T,\alpha}(x^n)- M_{T,\alpha}(x)|\to 0$, since $\alpha \notin U(x,T)$ by assumption.
\end{proof}

\begin{corollary} 
	\label{stopping_quantile_continuity_C}
	Let $\alpha \in (0,1),T>0,\gamma \in \mathbb{R}^d$. Suppose that $x \in C_{\mathbb{R}^d}[0,\infty)$; then we have
	\begin{equation}
	\tau_{M_{T,\alpha}}(x)=\inf\{t:\gamma \cdot x_t=M_{T,\alpha}(x)\}
	\end{equation}
	If $x^n\to x$ in the Skorokhod topology and:
	\begin{enumerate}
		\item $M_{T,\alpha}(x)\neq \gamma \cdot x_0$;
		\item Equivalently, $y\mapsto \frac{1}{T}\int_0^T\mathbf{1}_{\{z:\gamma \cdot z\leq y\}}(x_s)ds$ is continuous or $\beta \mapsto M_{T,\beta}(x)$ is strictly increasing;
		\item $\alpha\in  \{\alpha:\lim_{\beta \to \alpha}\tau_{M_{T,\beta}}(x)=\tau_{M_{T,\alpha}}(x)\}$
	\end{enumerate}
	then $(M_{T,\alpha}(x^n),\tau_{M_{T,\alpha}}(x^n))\to (M_{T,\alpha}(x),\tau_{M_{T,\alpha}}(x))$.
\end{corollary}
\begin{proof} This immediately follows from Theorem \ref{stopping_quantile_continuity} and the fact that $U(x,T)=\emptyset$ due to the continuity of $x$ by Lemma \ref{continuity_set_continuous_paths}, which then also implies that $\beta \mapsto M_{T,\beta}(x)$ is continuous and strictly increasing.
\end{proof}

\begin{lemma} \textcolor{black}{
	} Let $T>0,\alpha \in (0,1),\gamma \in \mathbb{R}^d,x \in \mathcal{D}_{\mathbb{R}^d}[0,\infty)$ and $\alpha \notin U(x,T)$ and $M_{T,\alpha}(x)\neq \gamma \cdot x_0$. Then
	\begin{equation}
	\lim_{\beta \to \alpha}\tau_{M_{T,\beta}}(x)=\tau_{M_{T,\alpha}}(x)\iff\begin{cases}
	\lim_{\beta \downarrow \alpha}\tau_{M_{T,\beta}}(x)=\tau_{M_{T,\alpha}}(x)&M_{T,\alpha}(x)>\gamma \cdot x_0\\
	\lim_{\beta \uparrow \alpha}\tau_{M_{T,\beta}}(x)=\tau_{M_{T,\alpha}}(x)&M_{T,\alpha}(x)<\gamma \cdot x_0\\
	\end{cases}
	\end{equation}
\end{lemma}
\begin{proof} Since $\alpha \notin U(x,T)$, then $M_{T,\beta}(x)\to M_{T,\alpha}(x)$ as $\beta \to \alpha$. The claim follows from the definition of $x\mapsto \tau_{M_{T,\alpha}}(x)$ as in Theorem \ref{stopping_quantile_continuity}: if $M_{T,\alpha}(x)> \gamma \cdot x_0$, then $\alpha \mapsto \tau_{M_{T,\alpha}}(x)=\inf\{t>0:\sup_{s\leq t}(\gamma \cdot x)_s\geq M_{T,\alpha}(x)\}$. Since $\nu\mapsto \inf\{t>0:\sup_{s\leq t}(\gamma \cdot x)_s\geq \nu\}$ is nondecreasing left-continuous, we have that $\lim_{\beta \to \alpha}\tau_{M_{T,\beta}}(x)=\tau_{M_{T,\alpha}}(x)$ if and only if $\lim_{\beta \downarrow \alpha}\tau_{M_{T,\beta}}(x)=\tau_{M_{T,\alpha}}(x)$. Similarly, if $M_{T,\alpha}(x)< \gamma \cdot x_0$, then $\alpha \mapsto \tau_{M_{T,\alpha}}(x)=\inf\{t>0:\inf_{s\leq t}(\gamma \cdot x)_s\leq M_{T,\alpha}(x)\}$. Since $\nu\mapsto \inf\{t>0:\inf_{s\leq t}(\gamma \cdot x)_s\leq \nu\}$ is nonincreasing càd, we have that $\lim_{\beta \to \alpha}\tau_{M_{T,\beta}}(x)=\tau_{M_{T,\alpha}}(x)$ if and only if $\lim_{\beta \uparrow \alpha}\tau_{M_{T,\beta}}(x)=\tau_{M_{T,\alpha}}(x)$.
\end{proof}

\section{Brownian hyperplane $\alpha$-quantile}

In this section we finally introduce the main propositions regarding the hyperplane $\alpha$-quantile applied to Brownian paths which will be needed to prove Theorem \ref{weak_convergence_BM_alpha_quantile_time}. The proof of these results rely upon explicit functional forms of the joint density of the Brownian time functional and the Brownian hyperplane $\alpha$-quantile which are found in \cite{Dassios_quantileBMstopping}; we will also mention distributional results of the marginal density of the Brownian $\alpha$-quantile which can be found in \cite{Dassios_quantileBM}. 
\textcolor{black}{We will first show a straightforward application of the aforementioned results for the expected value of the Brownian hyperplane $\alpha$-quantile.}\textcolor{black}{ 
	\begin{example} Let $\gamma \in \mathbb{R}^d$. Suppose $X=\Sigma W$ where $W$ is a standard $\mathbb{R}^d$-Brownian motion and $\Sigma$ is such that $\Sigma \Sigma'$ is a variance-covariance matrix and let $\alpha \in (0,1)$. By a.s. continuity of Brownian paths, Theorem \ref{weak_convergence_alpha_quantile} holds. We also have that
		\begin{equation}
		E[M_{t,\alpha}(X)]=\|\gamma'\Sigma\|\frac{\sqrt{2\alpha t}-\sqrt{2(1-\alpha) t}}{\sqrt{\pi}}
		\end{equation}
		and $M_{t,\beta}(X)\to^{L^1}M_{t,\alpha}(X)$ as $\beta \downarrow \alpha$.
	\end{example}
	\begin{indeed} Note that $\gamma \cdot X$ is a (scaled) one-dimensional Brownian motion in the filtration of $W$ with constant volatility parameter $\|\gamma'\Sigma\|$; this can be seen e.g. by Lévy characterisation. Define $\sigma:=\|\gamma' \Sigma\|$. Denote with $\overline{Y}_s:=\sup_{u\leq s}(\gamma \cdot X)_u$ and $\underline{Y}_s:=\inf_{u\leq s}(\gamma \cdot X)_u$. By symmetry and reflection principle we have $E[\overline{Y}_s]=-E[\underline{Y}_s]=\sigma \sqrt{2s/\pi},\,\forall s \geq 0$. Now, also note again that 
		$\mathbf{1}_{\{z:\gamma \cdot z\leq y\}}(X_s)=\mathbf{1}_{(-\infty,y]}(\gamma \cdot X_s)$; by (\cite{Dassios_quantileBM}; Th.2. p. 390), we have $M_{t,\beta}(X)\sim \overline{Y}^{(1)}_{\beta t}+\underline{Y}^{(2)}_{(1-\beta) t}$ where $\overline{Y}_s^{(1)}:=\sup_{u\leq s}(\gamma \cdot X^{(1)})_u$, $\underline{Y}_s^{(2)}:=\inf_{u\leq s}(\gamma \cdot X^{(2)})_u$ and $X^{(1)},X^{(2)}$ are IID copies of $X$. The first claim follows. By using the fact that $\beta \mapsto M_{t,\beta}(X)$ is nondecreasing, we get
		\begin{eqnarray}
		E[|M_{t,\beta}(X)-M_{t,\alpha}(X)|]&=&E[M_{t,\beta}(X)]-E[M_{t,\alpha}(X)]\nonumber\\
		&=&\sigma\frac{\sqrt{2\beta t}-\sqrt{2(1-\beta) t}-\sqrt{2\alpha t}+\sqrt{2(1-\alpha) t}}{\sqrt{\pi}}\nonumber\\
		&\stackrel{\beta \downarrow \alpha}{\to}&0
		\end{eqnarray}
		and the second claim follows.
\end{indeed}}

\textcolor{black}{For simplicity of exposition,} in the following, $X=\Sigma W$ where $\Sigma$ is such that $\Sigma \Sigma'$ is a (nondegenerate) variance-covariance matrix and $W$ is a standard $\mathbb{R}^d$-Brownian motion.

\begin{proposition} 
	\label{BM_joint_continuity_occupation_time}	
	Let $T>0,\gamma \in \mathbb{R}^d/\{0\}$. Then $y\mapsto \frac{1}{T}\int_0^T\mathbf{1}_{\{z:\gamma \cdot  z\leq y\}}(X_s)ds$ is $P$-a.s. continuous. Therefore, equivalently, $(0,1)\ni \alpha \mapsto M_{T,\alpha}(X)$ is strictly increasing $P$-a.s.
\end{proposition}
\begin{proof} We know that $\gamma \cdot X$ is a Brownian motion in the filtration of $W$ with volatility parameter $\|\gamma'\Sigma\|$. Without loss of generality let $\|\gamma'\Sigma\| =1$. The local time $y \mapsto \lim_{\varepsilon \to 0}\frac{1}{2\varepsilon}\int_0^T\mathbf{1}_{(y-\varepsilon,y+\varepsilon]}(\gamma \cdot X_s(\omega))ds$ exists and is finite for all $y$ for $P$-almost always $\omega \in \Omega$ (see e.g. (\cite{KaratzasShreve}; 6.5. p. 203)), and therefore $y\mapsto \frac{1}{T}\int_0^T\mathbf{1}_{\{z:\gamma \cdot  z\leq y\}}(X_s)ds$, since it is càd, is continuous $P$-a.s. in $y$. The final claim follows from Proposition \ref{McNeiletal}.
\end{proof}

The subsequent proposition makes it clear that the previous result, while intuitive, is not easily proved with standard techniques and our argument necessarily relies on a deep property of the local times of Brownian motion.

\begin{proposition}{(\cite{Sato}; 21.2. p. 135)}
	\label{piecewise_constant_Poisson}	
	If $Y$ is a Lévy process, then $t\mapsto Y_t(\omega)$ is piecewise constant for $P$-almost always $\omega \in \Omega$ if and only if $Y$ is a compound Poisson process or $Y=0$ $P$-a.s.
\end{proposition}

\begin{proposition} 
	\label{counter_joint_continuity_occupation_time_2}	
	\textcolor{black}{
	} Let $T>0,\gamma \in \mathbb{R}^d$. Let $D=\{b_1,b_2,...\}\subset \mathbb{R}$ be a countable dense subset. Then the condition $P(\cup_{k \in \mathbb{N}}\{\lambda(\{s\in [0,T]:\gamma \cdot Y_s=b_k\})>0\})=0$ alone does not imply that $y\mapsto \frac{1}{T}\int_0^T\mathbf{1}_{\{z:\gamma \cdot z\leq y\}}(Y_s)ds$ is $P$-a.s. continuous.
\end{proposition}

\begin{proof} If $\gamma \cdot Y$ is compound Poisson process-distributed with a.s. nonnegative increments and with Lebesgue absolutely continuous density at all $t$ then it is a counterexample which shows the claim. Indeed, for any $b_k \in D$ we have $E[\lambda(\{s\in [0,T]:\gamma \cdot Y_s=b_k\})]=0$ because $P(\gamma \cdot Y_t=b_k)=0,\forall t>0,\forall k \in \mathbb{N}$ so it follows that
	\begin{equation}P(\cup_{k \in \mathbb{N}}\{\lambda(\{s\in [0,T]:\gamma \cdot Y_s=b_k\})>0\})\leq \sum_{k \in \mathbb{N}}P(\lambda(\{s\in [0,T]:\gamma \cdot Y_s=b_k\})>0)=0
	\end{equation}
	But $\gamma \cdot Y$ is a.s. piecewise constant, so $y\mapsto \frac{1}{T}\int_0^T\mathbf{1}_{\{z:\gamma \cdot z\leq y\}}(Y_s)ds$ is a.s. discontinuous in at least one point in $y$. Indeed, suppose $x\in \mathcal{D}[0,\infty)$ and $x$ is nondecreasing piecewise constant, that is there exist $0<t_1<...<t_k<...$ such that $x$ is constant on $[t_k,t_{k+1})$. Without loss of generality suppose $x$ has at least two consecutive jump-times $0<u_1<u_2<T$.
	So $x_{u_1}=x_{u_1^-}+\Delta x_{u_1}$, and then $\lim_{y \uparrow x_{u_1}}\int_{0}^T\mathbf{1}_{(-\infty,y]}(x_s)ds<\int_{0}^T\mathbf{1}_{(-\infty,x_{u_1}]}(x_s)ds$ because $x_s=x_{u_1}$ for all $s \in [u_1,u_2)$ and so $\int_{0}^T\mathbf{1}_{(-\infty,y]}(x_s)ds$ is constant for all $x_{u^-}\leq y<x_{u_1}$. We conclude.
\end{proof}

\begin{lemma} 
	\label{BM_as_nonzero_alpha_quantile}	
	Let $T>0,\gamma \in \mathbb{R}^d/\{0\},z \in \mathbb{R}$. Then, we have $P(M_{T,\alpha}(X)=z)=0,\,\forall \alpha \in (0,1)$.
\end{lemma}
\begin{proof} Recall $\gamma \cdot X$ is a one dimensional (scaled by constant volatility $\|\gamma'\Sigma\|$) Brownian motion in the filtration of $W$. The claim immediately follows from the fact that $M_{T,\alpha}(X)$ has a Lebesgue density for any $\alpha\in (0,1),T>0$ (see e.g. (\cite{Dassios_quantileBM}; Th.1. p. 390)).
\end{proof}

\begin{proposition}
	\label{BM_continuity_random_times}
	Let $\alpha \in (0,1),T>0,\gamma \in \mathbb{R}^d/\{0\}$; then $P(\lim_{\beta \to \alpha}\tau_{M_{T,\beta}}(X)=\tau_{M_{T,\alpha}}(X))=1$.
\end{proposition}
\begin{proof} Of course $X_0=0$ a.s. Since $X$ has a.s. continuous paths, then $\beta \mapsto M_{T,\beta}(X)$ is a.s. strictly increasing continuous by Proposition \ref{BM_joint_continuity_occupation_time} and Lemma \ref{continuity_set_continuous_paths}. On $\{M_{T,\alpha}(X)>0\}$, since the function $[0,\infty)\ni\nu \mapsto \inf\{t>0:\sup_{s\leq t}(\gamma \cdot X)_s\geq \nu\}$ is nondecreasing and left-continuous, we have $\lim_{\beta \uparrow \alpha}\tau_{M_{T,\beta}}(X)=\tau_{M_{T,\alpha}}(X)$. Now note, for $\beta>\alpha$, that:
	\begin{eqnarray}&P(\{\tau_{M_{T,\beta}}(X)>\tau_{M_{T,\alpha}}(X)+\varepsilon\}\cap\{M_{T,\alpha}(X)>0\})\nonumber \\
	&\leq  \varepsilon^{-1}(E[\mathbf{1}_{\{M_{T,\alpha}(X)>0\}}\tau_{M_{T,\beta}}(X)]-E[\mathbf{1}_{\{M_{T,\alpha}(X)>0\}}\tau_{M_{T,\alpha}}(X)])\nonumber \\
	&\leq \varepsilon^{-1}(E[\mathbf{1}_{\{M_{T,\beta}(X)>0\}}\tau_{M_{T,\beta}}(X)]-E[\mathbf{1}_{\{M_{T,\alpha}(X)>0\}}\tau_{M_{T,\alpha}}(X)])
	\end{eqnarray}
	and by (\cite{Dassios_quantileBMstopping}; Remark.1. p. 32) we have by Tonelli-Fubini (without loss of generality $T=1$ and $\gamma \cdot X$ has unit volatility parameter)
	\begin{eqnarray}E[\mathbf{1}_{\{M_{1,\beta}(X)>0\}}\tau_{M_{1,\beta}}(X)]&=&\int_{{\mathbb{R}^+}}\int_{(0,\beta]}u\frac{b}{\pi \sqrt{u^3(1-u)}}e^{-b^2/(2u)}dudb\nonumber\\
	&=&\int_{(0,\beta]}\frac{1}{\pi}\sqrt{\frac{u}{1-u}}du\nonumber\\
	&=&\frac{1}{\pi}(\arcsin(\sqrt{\beta})-\sqrt{(1-\beta)\beta})
	\end{eqnarray}
	But this function is continuous in $\beta$. So $E[\mathbf{1}_{\{M_{T,\beta}(X)>0\}}\tau_{M_{T,\beta}}(X)]\to E[\mathbf{1}_{\{M_{T,\alpha}(X)>0\}}\tau_{M_{T,\alpha}}(X)]$ and so $P(\{\tau_{M_{T,\beta}}(X)>\tau_{M_{T,\alpha}}(X)+\varepsilon\}\cap\{M_{T,\alpha}(X)>0\})\to 0,\,\forall \varepsilon >0$. Therefore since $\nu\mapsto \inf\{t>0:\sup_{s\leq t}(\gamma \cdot x)_s\geq \nu\}$ is nondecreasing left-continuous we have
	\begin{equation}P(\{\lim_{\beta \downarrow \alpha}\tau_{M_{T,\beta}}(X)>\tau_{M_{T,\alpha}}(X)+\varepsilon\}\cap\{M_{T,\alpha}(X)>0\})=0,\forall \varepsilon >0\end{equation}
	which implies, as $\varepsilon \downarrow 0$ that
	\begin{equation}P(\{\lim_{\beta \downarrow \alpha}\tau_{M_{T,\beta}}(X)>\tau_{M_{T,\alpha}}(X)\}\cap\{M_{T,\alpha}(X)>0\})=0
	\end{equation}
	and therefore $P(\{\lim_{\beta \downarrow \alpha}\tau_{M_{T,\beta}}(X)\neq \tau_{M_{T,\alpha}}(X)\}\cap\{M_{T,\alpha}(X)>0\})=0$. We then conclude that $\lim_{\beta \to \alpha}\tau_{M_{T,\beta}}(X)=\tau_{M_{T,\alpha}}(X)$ on $\{M_{T,\alpha}(X)>0\}$ with probability one. A symmetric argument holds on $\{M_{T,\alpha}(X)<0\}$: indeed for $\beta<\alpha$ (and so $M_{T,\alpha}<0\implies M_{T,\beta}<0$) we again have
	\begin{eqnarray}&P(\{\tau_{M_{T,\beta}}(X)>\tau_{M_{T,\alpha}}(X)+\varepsilon\}\cap\{M_{T,\alpha}(X)<0\})\nonumber \\
	&\leq \varepsilon^{-1}(E[\mathbf{1}_{\{M_{T,\beta}(X)<0\}}\tau_{M_{T,\beta}}(X)]-E[\mathbf{1}_{\{M_{T,\alpha}(X)<0\}}\tau_{M_{T,\alpha}}(X)])
	\end{eqnarray}	
	and we argue similarly as above. Therefore we conclude since $P(M_{T,\alpha}(X)=0)=0$ by Lemma \ref{BM_as_nonzero_alpha_quantile}.
\end{proof}

\textcolor{black}{
	Ideally, we would like to extend the results of this section to processes which generalise Brownian motion, such as SDEs of the form $dY_t=b(Y_t)dW_t$ where $b$ is a $d\times d$ time-homogeneous volatility matrix. However, the direct extension of, in particular, Proposition \ref{BM_continuity_random_times} does not appear immediate, since explicit distributional results of $(M_{T,\alpha}(Y),\tau_{M_{T,\alpha}}(Y))$ in such case are not - to the best of our knowledge - known in the literature. However, we conjecture that the equivalent of Theorem \ref{weak_convergence_BM_alpha_quantile_time} for $Y$ is true under standard technical assumptions on $b$.
}

\end{document}